\documentclass[12pt,leqno]{amsart}
\usepackage{amsmath}
\usepackage{amsthm}
\usepackage{amssymb}
\def\overset#1#2{{\mathrel{\mathop {{#2}_{}}\limits^{#1}}}}
\def\underset#1#2{{\mathrel{\mathop {{}_{} {#2}}\limits_{{#1}_{}}}}}
\def\upplim_#1{\underset{#1}{\overline\lim}\;}
\def\lowlim_#1{\underset{#1}{\underline\lim}\;}
\newtheorem{claim}[equation]{\indent \rm {\it Claim}}

\newtheorem{lemma}[equation]{Lemma}

\newtheorem{theorem}[equation]{Theorem}

\newcommand{\C}{{\mathbf{C}}}

\renewcommand{\P}{{\mathbf{P}}}

\newcommand{\zero}{\mathrm{Zero}}

\newcommand{\supp}{\mathrm{Supp}\,}

\newcommand{\Z}{\mathbf{Z}}
\numberwithin{equation}{section}

\begin{document}
\title[Degeneracy and finiteness theorem for meromorphic mappings]{Degeneracy and finiteness theorems for meromorphic mappings in several complex variables} 

\author{Si Duc Quang}

\address{Department of Mathematics, Hanoi National University of Education\\
136-Xuan Thuy, Cau Giay, Hanoi, Vienam.}
\email{quangsd@hnue.edu.vn}

\def\thefootnote{\empty}
\footnotetext{
2010 Mathematics Subject Classification:
Primary 32H30, 32A22; Secondary 30D35.\\
\hskip8pt Key words and phrases: second main theorem, uniqueness problem, meromorphic mapping, truncated multiplicity.}

\begin{abstract} {In this article, we prove that there are at most two meromorphic mappings of $\C^m$ into $\P^n(\C)\ (n\geqslant 2)$ sharing $2n+2$ hyperplanes in general position regardless of multiplicity, where all zeros with multiplicities more than certain values do not need to be counted. We also show that  if three meromorphic mappings $f^1,f^2,f^3$ of $\C^m$ into $\P^n(\C)\ (n\geqslant 5)$ share $2n+1$ hyperplanes in general position with truncated multiplicity then the map $f^1\times f^2\times f^3$ is linearly degenerate.}
\end{abstract}
\maketitle

\section{Introduction}

In $1926$, R. Nevanlinna \cite{N} showed that two distinct nonconstant meromorphic functions $f$ and $g$ on the complex plane $\C$  cannot have the same inverse images for five distinct values, and that $g$ is a special type of linear fractional transformation of $f$ if they have the same inverse images counted with multiplicities for four distinct values \cite{N}. These results are usually called the five values and the values theorems of R. Nevanlinna. 
 
After that, many authors extended and improved the results of Nevanlinna to the case of meromorphic mappings into complex projective sapces. The extensions of the five values theorem are usually called the uniqueness theorems, and the extensions of the four values theorem are usually called the finiteness theorems. Here we formulate some recent results on this problem.

To state some of them, first of all we recall the following.

Let $f$ be a nonconstant meromorphic mapping of $\C^m$ into $\P^n(\C)$ and  $H$ a hyperplane in $\P^n(\C).$ Let 
$k$ be a positive integer or $k=\infty$. 
Denote by $\nu_{(f,H)}$ the map of $\C^m$ into $\Z$ whose value  $\nu_{(f,H)}(a)\ (a\in\C^m)$ is the intersection 
multiplicity of the images of $f$ and $H$ at $f(a).$ For every $z\in \C^m$, we set 
\begin{align*}
&\nu_{(f,H),\leqslant k}(z) =
\begin{cases}
0 & \text { if } \nu_{(f,H)}(z)>k,\\
\nu_{(f,H)}(z)& \text { if } \nu_{(f,H)}(z)\leqslant k,
\end{cases}\\
\text{ and }&\nu_{(f,H),\geqslant k}(z) =
\begin{cases}
0 & \text { if } \nu_{(f,H)}(z)<k,\\
\nu_{(f,H)}(z)& \text { if } \nu_{(f,H)}(z)\geqslant k,
\end{cases}
\end{align*}

Take a meromorphic mapping $f$ of $\C^m$ into $\P^n(\C)$ which is linearly nondegenerate over $\C$, a positive integer $d$ and $q$ hyperplanes $H_1,\ldots ,H_q$ of $\P^n(\C)$ in general position with
$$\dim  f^{-1}(H_i \cap H_j) \leqslant m-2 \quad (1 \leqslant i<j \leqslant q)$$
and consider the set $\mathcal {F}(f,\{H_i\}_{i=1}^q,d)$ of all linearly nondegenerate over $\C$ meromorphic maps $g: \C^m \to \P^n(\C)$ satisfying the following two conditions:

\ (a)\ $\min\ (\nu_ {(f,H_j)},d)=\min\ (\nu_{(g,H_j)},d)\quad (1\leqslant j \leqslant q),$

\ (b) \ $f(z) = g(z)$ on $\bigcup_{j=1}^{q}f^{-1}(H_j)$.\\
We see that conditions a) and b) mean the sets of all intersecting points (counted with multiplicity to level $d$) of $f$ and $g$ with each hyperplane are the same, and two mappings $f$ and $g$ agree on these sets. If $d=1$, we will say that $f$ and $g$ share $q$ hyperplanes $\{H_j\}_{j=1}^q$ regardless of multiplicity. 

\vskip0.2cm
Denote by  $\sharp\ S$ the cardinality of the set  $S.$ In 1983, L. Smiley \cite{S} proved the following uniqueness theorem.

\vskip0.2cm
\noindent
{\bf Theorem A.} {\it \ If $q=3n+2$ then $\sharp\ \mathcal {F}(f,\{H_i\}_{i=1}^q,1)=1.$}

\vskip0.2cm 
\noindent
In 1998, H. Fujimoto \cite{Fu98} proved a finiteness theorem for meromorphic mappings as follows.

\vskip0.2cm
\noindent
{\bf Theorem B.} {\it \ If $q=3n+1$ then $\sharp\ \mathcal {F}(f,\{H_i\}_{i=1}^q,2)\leqslant 2.$}

In 2009, Z. Chen-Q. Yan \cite{CY} considered the case of $2n+3$ hyperplanes and proved the following uniqueness theorem.

\vskip0.2cm
\noindent
{\bf Theorem C.} {\it \ If $q=2n+3$ then $\sharp\ \mathcal {F}(f,\{H_i\}_{i=1}^q,1)=1.$}

After that, in 2011 S. D. Quang \cite{SQ11} improved the result of Z. Chen-Q. Yan by omitting all zeros with multiplicity more than a certain number in the conditions on sharing hyperplanes of meromorphic mappings. As far as we known, there is still no uniqueness theorem for meromorphic mappings sharing less than $2n+3$ hyperplanes regardless of multiplicities. In 2011 Q. Yan-Z. Chen \cite{YC} also proved a degeneracy theorem as follows.

\vskip0.2cm
\noindent
{\bf Theorem D.} {\it \ If  $q=2n+2$ then the map $f^1\times f^2\times f^3$ of $\C^m$ into $\P^N(\C)\times \P^N(\C)\times \P^N(\C)$ is linearly degenerate for every  three maps $f^1,f^2,f^3\in\mathcal {F}(f,\{H_i\}_{i=1}^q,2)$.}

The first finiteness theorem for the case of meromorphic mappings sharing $2n+2$ hyperplanes regardless of multiplicities are given by S. D. Quang \cite{SQ12} in 2012 as follows.

\vskip0.2cm
\noindent
{\bf Theorem E.} {\it \ If $n\geqslant 2$ and $q=2n+2$ then $\sharp\ \mathcal {F}(f,\{H_i\}_{i=1}^q,1)\leqslant 2.$}

\vskip0.2cm
However we note that there is a gap in the proof of \cite[Theorem 1.1]{SQ12}. For detail, the inequality (3.26) in \cite[Lemma 3.20]{SQ12} does not holds. Hence the inequality of \cite[Lemma 3.20(ii)]{SQ12} may not hold. In order to fix this gap, we need to slightly change the estimate of this inequality by adding $N^{(1)}_{(f,H_j)}(r)$ to its right-hand side. The rest of the proof is still valid for the case where $N\geqslant 3$. In this paper, we will show a correction for \cite[Lemma 3.20]{SQ12} (see Lemma \ref{4.8} below). Also this theorem (including the case where $N=2$) will be corrected and improved (see Theorem \ref{1.2} below) by another approach.

We would also like to emphasize that in the above results, all intersecting points of the mappings and the hyperplanes are considered. It seems to us that the technique used in the proof of the above results do not work for the case where all such points with multiplicities more than a certain number are not taken to count. Our first purpose in this paper is to improve the above result by omitting all such intersecting points. In order to states the main results, we give the following definition.

Let $f$ be a linearly nondegenerate meromorphic mapping of $\C^m$ into $\P^n(\C)$ and let $H_1,\ldots ,H_q$ be $q$ hyperplanes  of $\P^n(\C)$ in general position. Let $k_1,\ldots ,k_q$ be $q$ positive integers or $+\infty$. Assume that
$$\dim \{z; \nu_{(f,H_i),\leqslant k_i}(z)\cdot\nu_{(f,H_j),\leqslant k_j}(z)> 0\} \leqslant m-2 \quad (1 \leqslant i<j \leqslant q).$$ 
Let $d$ be an integer. We consider the set $\mathcal {F}(f,\{H_i,k_i\}_{i=1}^q,d)$ of all meromorphic maps 
$g: \C^m \to \P^n(\C)$ satisfying the conditions:
\begin{itemize}
\item[(a)] $\min\ (\nu_ {(f,H_i),\leqslant k_i},d)=\min\ (\nu_{(g,H_i),\leqslant k_i},d)\quad (1\leqslant j \leqslant q),$
\item[(b)] $f(z) = g(z)$ on $\bigcup_{i=1}^{q}\{z;\nu_{(f,H_i),\leqslant k_i}(z)>0\}$.
\end{itemize}
Then we see that $\mathcal{F}(f,\{H_i\}_{i=1}^q,d)=\mathcal{F}(f,\{H_i,\infty\}_{i=1}^q,d)$

\begin{theorem}\label{1.2}
Let $f$ be a linearly nondegenerate meromorphic mapping of $\C^m$ into $\P^n(\C )$ $(n\geqslant 2)$. Let $H_1,\ldots ,H_{2n+2}$ be $2n+2$ hyperplanes of $\P^n(\C )$ in general position and let $k_1,\ldots ,k_{n+2}$ be positive integers or $+\infty$. Assume that
$$\dim \{z; \nu_{(f,H_i),\leqslant k_i}(z)\cdot\nu_{(f,H_j),\leqslant k_j}(z)> 0\} \leqslant m-2 \quad (1 \leqslant i<j \leqslant 2n+2),$$ 
$$\text{ and }  \sum_{i=1}^{2n+2}\dfrac{1}{k_i+1}< \min\left\{\dfrac{n+1}{3n^2+n},\dfrac{5n-9}{24n+12},\dfrac{n^2-1}{10n^2+8n}\right\}.$$
Then $\sharp \mathcal{F}(f,\{H_i,k_i\}_{i=1}^{2n+2},1)\leqslant 2.$
\end{theorem}
Then we see that in the case $n\geqslant 2$, Theorems D and E are corollaries of Theorem \ref{1.2} when $k_1=\cdots =k_{2n+2}=+\infty$.

The last purpose of this paper is to prove a degeneracy theorem for three mappings sharing $2n+1$ hyperplanes. Namely, we will proved the following.
\begin{theorem}
Let $f$ be a linearly nondegenerate meromorphic mapping of $\C^m$ into $\P^n(\C )\ (n\geqslant 5)$. Let $H_1,\ldots ,H_{2n+1}$ be $2n+1$ hyperplanes of $\P^n(\C )$ in general position and let $k_1,\ldots ,k_{2n+1}$ be positive integers or $+\infty$ such that
$$\dim \{z; \nu_{(f,H_i),\leqslant k_i}(z)\cdot\nu_{(f,H_j),\leqslant k_j}(z)> 0\} \leqslant m-2 \quad (1 \leqslant i<j \leqslant 2n+2).$$
 If there exists a positive integer $p$ with $p\leqslant n$ and
$$\sum_{i=1}^{2n+1}\dfrac{1}{k_i+1}<\dfrac{np-3n-p}{4n^2+3np-n}.$$
then the map $f^1\times f^2\times f^3$ of $\C^m$ into $\P^n(\C)\times \P^n(\C)\times \P^n(\C)$ is linearly degenerate for every three maps $f^1,f^2,f^3\in\mathcal {F}(f,\{H_i,k_i\}_{i=1}^{2n+1},p)$
\end{theorem}

\section{Basic notions in  Nevanlinna theory}

\subsection{Counting functions of divisors.}
We set $||z|| = \big(|z_1|^2 + \dots + |z_m|^2\big)^{1/2}$ for
$z = (z_1,\dots,z_n) \in \C^m$ and define 
$$B(r) := \{ z \in \C^m : ||z|| < r\},\quad S(r) := \{ z \in \C^m : ||z|| = r\}\ (0<r<\infty).$$
Define
$$v_{m-1}(z) := \big(dd^c ||z||^2\big)^{m-1}\quad \quad \text{and}$$
$$\sigma_m(z):= d^c \text{log}||z||^2 \land \big(dd^c
\text{log}||z||^2\big)^{m-1}
 \text{on} \quad \C^m \setminus \{0\}.$$

We mean by a divisor divisor $\nu$ on a domain $\Omega$ in $\C^m$  a formal sum
$$ \nu=\sum_{\lambda\in \Lambda}a_\lambda Z_{\lambda}, $$
where $a_{\lambda}\in \Z$ and $\{Z_{\lambda}\}_{\lambda\in\Lambda}$ is a  locally finite family of distinct irreducible hypersurfaces of $\Omega$. Then, we may consider the divisor $\nu$ as a function on $\Omega$ with values in $\Z$ as follows
$$ \nu (z)=\sum_{Z_{\lambda}\ni z}a_{\lambda}. $$
 The support of $\nu$ is defined by $\supp\nu =\bigcup_{a_{\lambda}\ne 0}Z_{\lambda}$.

For a nonzero meromorphic function $\varphi$ on a domain $\Omega$ in
$\C^m$, we denote by $\nu^0_{\varphi}$ (resp. $\nu^\infty_{\varphi}$) the divisor of zeros (resp. divisor of poles) of $\varphi$, and denote by $\nu_\varphi =\nu^0_{\varphi}-\nu^\infty_{\varphi}$ the divisor generated by $\varphi$.

For a divisor $\nu$ on $\C^m$ and for positive integers $k,M$ (or
$M= \infty$), we define the counting functions of $\nu$ as follows.
Set
$$\nu^{(M)}(z)=\min\ \{M,\nu(z)\},$$
\begin{align*}
\nu_{\leqslant k}^{(M)}(z) =
\begin{cases}
0 & \text { if } \nu(z)>k,\\
\nu^{(M)}(z)& \text { if } \nu(z)\leqslant k,
\end{cases}
\end{align*}
\begin{align*}
\nu_{>k}^{(M)}(z) =
\begin{cases}
\nu^{(M)}(z)& \text { if } \nu(z)> k,\\
0 & \text { if } \nu(z)\leqslant k.
\end{cases}
\end{align*}

We define $n(t)$ by
\begin{align*}
n(t) =
\begin{cases}
\int\limits_{|\nu|\,\cap B(t)}
\nu(z) v_{n-1} & \text  { if } n \geqslant 2,\\
\sum\limits_{|z|\leqslant t} \nu (z) & \text { if }  n=1.
\end{cases}
\end{align*}

Similarly, we define \quad $n^{(M)}(t), \ n_{\leqslant k}^{(M)}(t), \
n_{>k}^{(M)}(t).$

Define
$$ N(r,\nu)=\int\limits_1^r \dfrac {n(t)}{t^{2n-1}}dt \quad (1<r<\infty).$$

Similarly, we define  \ $N(r,\nu^{(M)}), \ N(r,\nu_{\leqslant k}^{(M)}), \
N(r,\nu_{>k}^{(M)})$ and denote them by \ $N^{(M)}(r,\nu)$, $N_{\leqslant k}^{(M)}(r,\nu)$, $N_{>k}^{(M)}(r,\nu)$ respectively.

Let $\varphi : \C^m \longrightarrow \C $ be a meromorphic function.
Define 
$$N_{\varphi}(r)=N(r,\nu^0_{\varphi}),\ N_{\varphi}^{(M)}(r)=N^{(M)}(r,\nu^0_{\varphi}),$$
$$N_{\varphi,\leqslant k}^{(M)}(r)=N_{\leqslant k}^{(M)}(r,\nu^0_{\varphi}),\ N_{\varphi,> k}^{(M)}(r)=N_{>k}^{(M)}(r,\nu^0_{\varphi}).$$

For brevity we will omit the superscript $^{(M)}$ if $M=\infty$.

For a set $S\subset\C^m$, we define the characteristic function of $S$ by
$$ \chi_S(z)=\begin{cases}
1&\text{ if }z\in S,\\
0&\text{ if }z\not\in S.
\end{cases} $$
If the closure $\bar{S}$ of $S$ is an analytic subset of $\C^m$ then we denote by $N(r,S)$ the counting function of the reduced divisor whose support is the union of all irreducible components of $\bar{S}$ with codimension one.

\subsection{Characteristic and Proximity functions.} Let $f : \C^m \longrightarrow \P^n(\C)$ be a meromorphic
mapping.
For arbitrarily fixed homogeneous coordinates
$(w_0 : \dots : w_n)$ on $\P^n(\C)$, we take a reduced representation
$f = (f_0 : \dots : f_n)$, which means that each $f_i$ is a
holomorphic function on $\C^m$ and
$f(z) = \big(f_0(z) : \cdots : f_n(z)\big)$ outside the analytic set
$\{ f_0 = \cdots = f_n= 0\}$ of codimension $\geqslant 2$.
Set $\Vert f \Vert = \big(|f_0|^2 + \dots + |f_n|^2\big)^{1/2}$.

The characteristic function of $f$ is defined by
\begin{align*}
T_f(r) = \int\limits_{S(r)} \text{log}\Vert f \Vert \sigma_m -
\int\limits_{S(1)}\text{log}\Vert f\Vert \sigma_m.
\end{align*}

Let $H$ be a hyperplane in $\P^n(\C)$ given by
$H=\{a_0\omega_0+\cdots +a_n\omega_n\},$ where $(a_0,\ldots ,a_n)\ne (0,\ldots ,0)$.
We set  $(f,H)=\sum_{i=0}^na_if_i$. Then we see that the divisor $\nu_{(f,H)}$ does not depend on the reduced representation of $f$ and presentation of $H$. We define the proximity function of $H$ by
$$m_{f,H}(r)=\int_{S(r)}\log \dfrac {||f||\cdot ||H||}{|(f,H)|}\sigma_m-\int_{S(1)}\log \dfrac {||f||\cdot ||H||}{|(f,H)|}\sigma_m,$$
where $||H||=(\sum_{i=0}^N|a_i|^2)^{\frac {1}{2}}.$

Let $\varphi$ be a nonzero meromorphic function on $\C^m$, which
are occasionally regarded as a meromorphic mapping into $\P^1(\C)$. The
proximity function of $\varphi$ is defined by
$$m(r,\varphi):=\int_{S(r)}\log \max\ (|\varphi|,1)\sigma_n.$$

As usual, by the notation ``$|| \ P$''  we mean the assertion $P$ holds for all $r \in [0,\infty)$ excluding a Borel subset $E$ of the
interval $[0,\infty)$ with $\int_E dr<\infty$.

\subsection{Some lemmas.}The following results play essential roles in Nevanlinna theory (see \cite{NO}).

\begin{theorem}[The first main theorem]\label{2.7}
Let $f: \C^m \to \P^n(\C)$ be a linearly nondegenerate meromorphic mapping and $H$ be a hyperplane in $\P^n(\C)$. Then 
$$N_{(f,H)}(r)+m_{f,H}(r)=T_f(r)\ (r>1).$$
\end{theorem}

\begin{theorem}[The second main theorem]\label{2.8}
Let $f: \C^m \to \P^n(\C)$ be a linearly nondegenerate meromorphic mapping and $H_1,\ldots ,H_q$ be hyperplanes in general position in $\P^n(\C).$ Then 
$$||\ \ (q-n-1)T_f(r) \leqslant \sum_{i=1}^qN_{(f,H_i)}^{(n)}(r)+o(T_f(r)).$$
\end{theorem}

For meromorphic functions $F,G,H$ on $\C^m$ and $\alpha =(\alpha_1,\ldots ,\alpha_m)\in \Z_+^m$, we put
$$
\Phi^\alpha(F,G,H):=F\cdot G\cdot H\cdot\left | 
\begin {array}{cccc}
1&1&1\\
\frac {1}{F}&\frac {1}{G} &\frac {1}{H}\\
\mathcal {D}^{\alpha}(\frac {1}{F}) &\mathcal {D}^{\alpha}(\frac {1}{G}) &\mathcal {D}^{\alpha}(\frac {1}{H})\\
\end {array}
\right|
$$
\vskip0.2cm
\noindent
\begin{lemma}[{\cite[Proposition 3.4]{Fu98}}]\label{4.3} If $\Phi^\alpha(F,G,H)=0$ and $\Phi^\alpha(\frac {1}{F},\frac {1}{G},\frac {1}{H})=0$ for all $\alpha$ with $|\alpha|\le1$, then one of the following assertions holds :

(i) \ $F=G, G=H$ or $H=F$

(ii) \ $\frac {F}{G},\frac {G}{H}$ and $\frac {H}{F}$ are all constant.
\end{lemma}

\begin{lemma}\label{new1}
Let $f^1,f^2,f^3$ be three maps in $\mathcal F(f,\{H_i,k_i\}_{i=1}^q,p)$. Assume that $f^i$ has a representation $f^i=(f^i_{0}:\cdots :f^i_{n})$, $1\leqslant i\leqslant 3$. Suppose that there exist $s,t,l\in\{1,\cdots ,q\}$ such that
$$ 
P:=Det\left (\begin{array}{ccc}
(f^1,H_s)&(f^1,H_t)&(f^1,H_l)\\ 
(f^2,H_s)&(f^2,H_t)&(f^2,H_l)\\
(f^3,H_s)&(f^3,H_t)&(f^3,H_l)
\end{array}\right )\not\equiv 0.
$$
Then we have
\begin{align*}
T(r)\geqslant&\sum_{i=s,t,l}(N(r,\min\{\nu_{(f^u,H_i),\leqslant k_i};1\leqslant u\leqslant 3\})\\
&-N^{(1)}_{(f,H_i),\le k_i}(r))+ 2\sum_{i=1}^qN^{(1)}_{(f,H_i),\le k_i}(r)+o(T(r)),
\end{align*}
where $T(r)=\sum_{u=1}^3T_{f^u}(r)$.
\end{lemma}
\begin{proof}
Denote by $S$ the closure of $\bigcup_{1\leqslant u\leqslant 3}I(f^u)\cup\bigcup_{1\leqslant i<j\leqslant 2n+2}\{z;\nu_{(f,H_i),\leqslant k_i}(z)\cdot\nu_{(f,H_j),\leqslant k_j}(z)>0\}$.
Then $S$ is an analytic subset of codimension two of $\C^m$. 

For $z\not\in S$, we consider the following two cases:

\textit{Case 1.} $z$ is a zero of $(f,H_{i})$ with multiplicity at most $k_i$, where $i\in\{s,t,l\}$. For instance, we suppose that $i=s$.
We set 
$$m=\min\{\nu_{(f^1,H_s),\leqslant k_s}(z),\nu_{(f^2,H_s),\leqslant k_s}(z),\nu_{(f^3,H_s),\leqslant k_s}(z)\}.$$ 
Then there exist a neighborhood $U$ of $z$ and a holomorphic function $h$ defined on $U$ 
such that $\zero (h)=U\cap\zero (f,H_{s})$ and $dh$ has no zero. Then the functions $\varphi_u=\frac{(f^u,H_s)}{h^m}\ (1\leqslant u\leqslant 3)$ are holomorphic in a neighborhood of $z$. On the other hand, since $f^1=f^2=f^3$ on $\supp \nu_{(f,H_{s}),\leqslant k_s}$, we have
$$ P_{uv}:=(f^u,H_{t})(f^v,H_{l})-(f^u,H_{l})(f^v,H_{t})=0\text{ on }\supp \nu_{(f,H_{s}),\leqslant k_s},\ 1\leqslant u<v\leqslant 3.$$
Therefore, there exist holomorphic functions $\psi_{uv}$ on a neighborhood of $z$ such that $P_{uv}=h\psi_{uv}.$
Then we have
$$ P=h^{m+1}(\varphi_1\psi_{23}-\varphi_2\psi_{13}+\varphi_3\psi_{12}) $$
on a neighborhood of $z$. This yeilds that
$$ \nu_P(z)\geqslant m+1= \sum_{i=s,t,l}(\min\{\nu_{(f^u,H_i),\leqslant k_i}(z);1\leqslant u\leqslant 3\}-\nu^{(1)}_{(f,H_i),\leqslant k_i}(z))+ 2\sum_{i=1}^q\nu^{(1)}_{(f,H_i),\leqslant k_i}(z).$$

\textit{Case 2.} $z$ is a zero point of $(f,H_i)$ with multiplicity at most $k_i$, where $i\not\in\{s,t,l\}$. There exist an index $v$ such that $(f^1,H_v)(z)\ne 0$. Since $f^1(z)=f^2(z)=f^3(z),$ we have $(f^u,H_v)(z)\ne 0\ (1\leqslant u\leqslant 3)$ and
 \begin{align*} 
P &= \prod_{u=1}^3(f^u,H_v)\cdot\det\left (
\begin{array}{ccc}
\dfrac{(f^1,H_s)}{(f^1,H_v)}&\dfrac{(f^1,H_t)}{(f^1,H_v)}&\dfrac{(f^1,H_{l})}{(f^1,H_{v})}\\
\dfrac{(f^2,H_1)}{(f^2,H_l)}&\dfrac{(f^2,H_t)}{(f^2,H_l)}&\dfrac{(f^2,H_{s})}{(f^2,H_{l})}\\
\dfrac{(f^3,H_1)}{(f^3,H_l)}&\dfrac{(f^3,H_t)}{(f^3,H_l)}&\dfrac{(f^3,H_{s})}{(f_3,H_{l})}
\end{array}\right )\\
\\
&=\prod_{u=1}^3(f^u,H_l)\cdot\det\left (
\begin{array}{ccc} 
\dfrac{(f^1,H_1)}{(f^1,H_l)}&\dfrac{(f^1,H_t)}{(f^1,H_l)}&\dfrac{(f^1,H_{s})}{(f^1,H_{l})}\\
\\
\frac{(f^2,H_1)}{(f^2,H_l)}-\frac{(f^1,H_1)}{(f^1,H_l)}&\frac{(f^2,H_t)}{(f^2,H_l)}-\frac{(f^1,H_t)}{(f^1,H_l)}&\frac{(f^2,H_s)}{(f^2,H_l)}-\frac{(f^1,H_s)}{(f^1,H_l)}\\
\\
\frac{(f^3,H_1)}{(f^3,H_l)}-\frac{(f^1,H_1)}{(f^1,H_l)}&\frac{(f^3,H_t)}{(f^3,H_l)}-\frac{(f^1,H_t)}{(f^1,H_l)}&\frac{(f^3,H_s)}{(f^3,H_l)}-\frac{(f^1,H_s)}{(f^1,H_l)}
\end{array}\right ).
\end{align*}
vanishes at $z$ with multiplicity at least two. Therefore, we have
$$ \nu_P(z)\geqslant 2= \sum_{i=s,t,l}(\min\{\nu_{(f^u,H_i),\leqslant k_i}(z);1\leqslant u\leqslant 3\}-\nu^{(1)}_{(f,H_i),\leqslant k_i}(z))+ 2\sum_{i=1}^q\nu^{(1)}_{(f,H_i),\leqslant k_i}(z).$$

Thus, from the above two cases we have 
\begin{align*}
\nu_P(z)\geqslant \sum_{i=s,t,l}(\min\{\nu_{(f^u,H_i),\leqslant k_i}(z);1\leqslant u\leqslant 3\}-\nu^{(1)}_{(f,H_i),\leqslant k_i}(z))+ 2\sum_{i=1}^q\nu^{(1)}_{(f,H_i),\leqslant k_i}(z),
\end{align*}
for all $z$ outside the analytic set $S$. Integrating both sides of the above inequality, we get
\begin{align*}
N_P(r)\geqslant&\sum_{i=s,t,l}(N(r,\min\{\nu_{(f^u,H_i),\leqslant k_i};1\leqslant u\leqslant 3\})-N^{(1)}_{(f,H_i),\leqslant k_i}(r))\\
&+ 2\sum_{i=1}^qN^{(1)}_{(f,H_i),\leqslant k_i}(r)+o(T(r)).
\end{align*}

On the other hand, by Jensen's formula and the definition of the characteristic function we have
\begin{align*}
N_{P}(r)=&\int_{S(r)}\log |P|\sigma_m + O(1)\\
\leqslant &\sum_{u=1}^3\int_{S(r)}\log (|(f^u,H_1)|^2+|(f^u,H_t)|^2+|(f^u,H_{s})|)^{\frac{1}{2}}\sigma_m +O(1)\\
\leqslant & \sum_{u=1}^3\int_{S(r)}\log ||f^u||\sigma_m +O(1)=T(r)+o(T(r)).
\end{align*}
Thus, we have
\begin{align*}
T(r)\geqslant&\sum_{i=s,t,l}(N(r,\min\{\nu_{(f^u,H_i),\leqslant k_i};1\leqslant u\leqslant 3\})-N^{(1)}_{(f,H_i),\leqslant k_i}(r))\\
&+ 2\sum_{i=1}^qN^{(1)}_{(f,H_i),\leqslant k_i}(r)+o(T(r)).
\end{align*}
The lemma is proved.
\end{proof}

\section{Proof of  Main Theorems}
Let $f$ be a linearly nondegenerate meromorphic mapping of $\C^m$ into $\P^n(\C )$. Let $H_1,\ldots ,H_{2n+2}$ be $2n+2$ hyperplanes of $\P^n(\C )$ in general position and let $k_i\geqslant n\ (1\leqslant i\leqslant 2n+2)$ be positive integers or $+\infty$ with
$$\dim \{z; \nu_{(f,H_i),\leqslant k_i}(z)\cdot\nu_{(f,H_j),\leqslant k_j}(z)> 0\} \leqslant m-2 \quad (1 \leqslant i<j \leqslant 2n+2).$$ 

In order to prove Theorem \ref{1.1}, we need the following lemmas.

\begin{lemma}\label{3.1} If $\sum_{i=1}^{2n+2}\frac{1}{k_i+1}<\frac{1}{n},$ then every mapping $g$ in $\mathcal F(f,\{H_i,k_i\}_{i=1}^{2n+2},1)$ is linearly nondegenerate and 
$$ || \ T_{g}(r)=O(T_{f}(r)) \mathrm{\ and\ } || \ T_{f}(r)=O(T_{g}(r)).$$
\end{lemma}
\begin{proof}
Suppose that there exists a hyperplane $H$ satisfying $g(\C^m)\subset H$. We assume that $f$ and $g$ have reduce representations $f=(f_{0}:\cdots :f_{n})$ and $g=(g_0:\cdots :g_n)$ respectively. Assume that $H=\{(\omega_0:\cdots :\omega_n)\ |\ \sum_{i=0}^na_i\omega_i=0\}$. 
Since $f$ is linearly nondegenerate, $(f.H)\not\equiv 0$. On the other hand $(f,H)(z)=(g,H)(z)=0$ for all $z\in\bigcup_{i=1}^{2n+2}\{\nu_{(f,H_i),\leqslant k_i}\}$, hence
$$ N_{(f,H)}(r)\geqslant \sum_{i=1}^{2n+2}N_{(f,H_i),\leqslant k_i}^{(1)}(r).$$
It yields that
\begin{align*}
||\ T_{f}(r)&\geqslant N_{(f,H)}(r)\geqslant \sum_{i=1}^{2n+2}N_{(f,H_i),\leqslant k_i}^{(1)}(r)=\sum_{i=1}^{2n+2}\bigl (N_{(f,H_i)}^{(1)}(r)-N_{(f,H_i),> k_i}^{(1)}(r)\bigl )\\ 
&\geqslant\sum_{i=1}^{2n+2}\dfrac{1}{n}N_{(f,H_i)}^{(n)}(r)-\sum_{i=1}^{2n+2}\dfrac{1}{k_i+1}T_{f}(r)\geqslant \big (\dfrac{n+1}{n}-\sum_{i=1}^{2n+2}\dfrac{1}{k_i+1}\big )T_f(r)+o(T_{f}(r)). 
\end{align*}
Letting $r\longrightarrow +\infty$, we get 
$$ \sum_{i=1}^{2n+2}\dfrac{1}{k_i+1}\geqslant\dfrac{1}{n}. $$
This is a contradiction. Hence $g(\C^m)$ can not be contained in any hyperplanes of $\P^n(\C)$. Therefore $g$ is linearly nondegenerate.

Also by the Second Main Theorem, we have 
\begin{align*}
||\quad (n+1) T_{g}(r)\leqslant &\sum_{i=1}^{2n+2} N_{(g,H_i)}^{(n)}(r)+o(T_{g}(r))\\
\leqslant & \sum_{i=1}^{2n+2}n\ N_{(g,H_i)}^{(1)}(r)+o(T_{g}(r))\\
= & \sum_{i=1}^{2n+2}n\bigl (N_{(g,H_i),\leqslant k_i}^{(1)}(r)+ N_{(g,H_i),>k_i}^{(1)}(r)\bigl )+o(T_{g}(r))\\
\leqslant& \sum_{i=1}^{2n+2}n\bigl (N_{(f,H_i),\leqslant k_i}^{(1)}(r)+\dfrac{1}{k_i+1}T_{g}(r)\bigl)+o(T_{g}(r))\\
\leqslant& \sum_{i=1}^{2n+2}n\bigl (T_{f}(r)+\dfrac{1}{k_i+1}T_{g}(r)\bigl)+o(T_{f}(r)+T_{g}(r)).
\end{align*}
Thus
$$ \bigl (n+1-\sum_{i=1}^{2n+2}\dfrac{n}{k_i+1}\bigl ) T_{g}(r)\leqslant n(2n+2)T_{f}(r)+o(T_{f}(r)+T_{g}(r)).$$
We note that
$$ n+1-\sum_{i=1}^{2n+2}\dfrac{n}{k_i+1}>n>0.$$
Hence $|| T_{g}(r)=O(T_{f}(r)).$ Similarly, we get $||  T_{f}(r)=O(T_{g}(r)).$
\end{proof}

\begin{lemma}\label{1.1}

Assume that $n\geqslant 2$ and
$$  \sum_{i=1}^{2n+2}\dfrac{1}{k_i+1}< \dfrac{n+1}{n(3n+1)}.$$
Then for three maps $f^1,f^2,f^3\in\mathcal{F}(f,\{H_i,k_i\}_{i=1}^{2n+2},1)$ we have $f^1\wedge f^2\wedge f^3=0.$
\end{lemma}

\begin{proof}
By Lemma \ref{3.1}, we have that $f^s$ is linearly nondegenerate and $|| T_{f^s}(r)=O(T_f(r))$ and $|| T_f(r)=O(T_{f^s}(r))$ for all $s=1,2,3.$

Suppose that $f^1\wedge f^2\wedge f^3\not\equiv 0$. For each $1\leqslant i\leqslant 2n+2$, we set 
$$ N_i(r)=\sum_{u=1}^3N^{(n)}_{(f^u,H_i),\leqslant k_i}(r)-(2n+1)N^{(1)}_{(f,H_i),\leqslant k_i}(r) .$$
Here, we note that for positive integers $a,b,c$ we have $(\min\{a,b,c\}-1)\geqslant\min\{a,n\}+\min\{a,n\}+\min\{a,n\}-2n-1.$ Then
$$ \min\{\nu_{(f^u,H_i),\leqslant k_i}(z);1\leqslant u\leqslant 3\}-\nu^{(1)}_{(f,H_i),\leqslant k_i}(z)\geqslant\sum_{u=1}^3\nu^{(n)}_{(f^u,H_i),\leqslant k_i}(z)-(2n+1)\nu^{(1)}_{(f,H_i),\leqslant k_i}(z) $$
for all $z\in\supp\nu_{(f,H_i),\leqslant k_i}$. This yeilds that
\begin{align*}
 N(r,\min\{\nu_{(f^u,H_i),\leqslant k_i}(z)&;1\leqslant u\leqslant 3\})-N^{(1)}_{(f,H_i),\leqslant k_i}(r)\\
&\geqslant\sum_{u=1}^3N^{(n)}_{(f^u,H_i),\leqslant k_i}(r)-(2n+1)N^{(1)}_{(f,H_i),\leqslant k_i}(r)=N_i(r).
\end{align*}

We denote by $\mathcal I$ the set of all permutations of the $(2n+2)-$tuple $(1,\ldots ,2n+2)$, that means
$$ \mathcal I=\{I=(i_1,\ldots ,i_{2n+2})\ :\ \{i_1,\ldots ,i_{2n+2}\}=\{1,\ldots ,{2n+2}\}\}. $$
For each $I=(i_1,\ldots ,i_{2n+2})\in\mathcal I$ we define a subset $E_I$ of $[1,+\infty )$ as follows
$$ E_I=\{r\geqslant 1\ :\ N_{i_1}(r)\geqslant\cdots\geqslant N_{i_{2n+2}}(r)\}. $$
It is clear that $\bigcup_{I\in\mathcal I}E_I=[1,+\infty ).$ Therefore, there exists an element of $\mathcal I,$ for instance it is $I_0=(1,2,\ldots ,2n+2)$, satisfying
$$ \int\limits_{E_{i_0}}dr=+\infty . $$
Then, we have $N_1(r)\geqslant N_2(r)\geqslant\cdots\geqslant N_{2n+2}(r)$ for all $r\in E_{i_0}.$

We consider $\mathcal M^{3}$ as a vector space over the field $\mathcal M$. For each $i=1,\ldots ,2n+2,$ we set 
$$ V_i=\left ((f^1,H_i),(f^2,H_i),(f^3,H_i)\right )\in \mathcal M^{3}. $$ 
We put
$$ s=\min\{i\ :\ V_1\wedge V_i\not\equiv 0\}. $$
Since $f^1\wedge f^2\wedge f^3\not\equiv 0$, we have $1<s<n+1.$ Also by again  $f^1\wedge f^2\wedge f^3\not\equiv 0$, there exists an index $t\in\{s+1,\ldots,n+1\}$ such that $V_1\wedge V_s\wedge V_t\not\equiv 0$. This means that
$$P:=\det (V_1,V_s,V_t)=\det\left (
\begin{array}{ccc}
(f^1,H_1)&(f^1,H_s)&(f^1,H_t)\\ 
(f^2,H_1)&(f^2,H_s)&(f^2,H_t)\\
(f^3,H_1)&(f^3,H_s)&(f^3,H_t)
\end{array}\right )\not\equiv 0. $$

Set $T(r)=\sum_{u=1}^3T_{f^u}(r)$. By Lemma \ref{new1}, for $r\in E_{I_0}$ we have

\begin{align*}
T(r)&\geqslant\sum_{i=1,s,t}(N(r,\min\{\nu_{(f^u,H_i),\leqslant k_i};1\leqslant u\leqslant 3\})-N^{(1)}_{(f,H_i),\leqslant k_i}(r))\\
&+ 2\sum_{i=1}^qN^{(1)}_{(f,H_i),\leqslant k_i}(r)+o(T(r))\\
&\geqslant N_1(r)+N_s(r)+2\sum_{i=1}^qN^{(1)}_{(f,H_i)}(r)+o(T(r))\\
&\geqslant\dfrac{1}{n+1}\sum_{i=1}^{2n+2}N_i(r)+2\sum_{i=1}^{2n+2}N_{(f,H_i),\leqslant k_i}^{(1)}(r)+o(T(r)).\\
&=\dfrac{1}{n+1}\sum_{i=1}^{2n+2}\biggl (\sum_{u=1}^3N^{(n)}_{(f^u,H_i),\leqslant k_i}(z)-(2n+1)N^{(1)}_{(f,H_i)}(z)\biggl )+ 2\sum_{i=1}^{2n+2}N_{(f,H_i),\leqslant k_i}^{(1)}(r)\\
&=\dfrac{1}{n+1}\sum_{i=1}^{2n+2}\sum_{u=1}^3N^{(n)}_{(f^u,H_i),\leqslant k_i}(z)+\dfrac{1}{3(n+1)}\sum_{i=1}^{2n+2}\sum_{u=1}^3N^{(1)}_{(f^u,H_i),\leqslant k_i}(r)\\
&\geqslant (1+\dfrac{1}{3n})\dfrac{1}{n+1}\sum_{i=1}^{2n+2}\sum_{u=1}^3N^{(n)}_{(f^u,H_i),\leqslant k_i}(r)\\
&\geqslant (1+\dfrac{1}{3n})\dfrac{1}{n+1}\sum_{i=1}^{2n+2}\sum_{u=1}^3\biggl (N^{(n)}_{(f^u,H_i)}(r)-N^{(n)}_{(f^u,H_i),> k_i}(r)\biggl )\\
&\geqslant (1+\dfrac{1}{3n})\dfrac{1}{n+1}\sum_{u=1}^3\biggl (n+1-\sum_{i=1}^{2n+2}\dfrac{n}{k_i+1}\biggl )T_{f^u}(r)+o(T(r))\\
&= \bigl (1+\dfrac{1}{3n}-\dfrac{3n+1}{3(n+1)}\sum_{i=1}^{2n+2}\dfrac{1}{k_i+1}\bigl)T(r)+o(T(r)).
\end{align*}
Letting $r\rightarrow +\infty$ $(r\in E_{i_0})$ we get 
\begin{align*}
1\geqslant 1+\dfrac{1}{3n}-\dfrac{3n+1}{3(n+1)}\sum_{i=1}^{2n+2}\dfrac{1}{k_i+1}.
\end{align*}
Thus
$$\sum_{i=1}^{2n+2}\dfrac{1}{k_i+1}\geqslant \dfrac{n+1}{n(3n+1)}. $$

This is a contradiction. Hence, $f^1\wedge f^2\wedge f^3\equiv 0.$ 
The lemma is proved. 
\end{proof}

Now for three mappings $f^1, f^2, f^3 \in \mathcal {F}(f,\{H_i,k_i\}_{i=1}^{2n+2},1)$, we define:
\begin{align*}
F_k^{ij}&=\dfrac {(f^k,H_i)}{(f^k,H_j)}\ (0\leqslant k\leqslant 2,\ 1\leqslant i,j\leqslant 2n+2),\\ 
V_i&=( (f^1,H_i),(f^2,H_i),(f^3,H_i))\in\mathcal M_m^3,
\end{align*}

\begin{align*}
T_i&=\{z;\nu_{(f,H_i),\leqslant k_i}(z)>0\}, S_i=\bigcup_{u=1}^3\{z;\nu_{(f_u,H_i),> k_i}(z)>0\},\\
R_i&=\bigcap_{u=1}^3\{z;\nu_{(f_u,H_i),> k_i}(z)>0\},\\
\nu_i&=\{z;k_i\geqslant\nu_{(f^u,H_i)}(z)\geqslant \nu_{(f^v,H_i)}(z)=\nu_{(f^t,H_i)}(z) \text{ for a permutation }(u,v,t)\text{ of }(1,2,3)\}.
\end{align*}
We write $V_i\cong V_j$ if $V_i\wedge V_j\equiv 0$, otherwise we write $V_i\not\cong V_j.$ For $V_i\not\cong V_j$, we wirte $V_i\sim V_j$ if there exist $1\leqslant u<v\leqslant 3$ such that $F_u^{ij}=F_v^{ij}$, otherwise we write $V_i\not\sim V_j$.

\begin{lemma}\label{4.1} With the assumption of Theorem \ref{1.2}. Let $h$ and $g$ be two elements of the family $\mathcal F(f,\{H_i,k_i\}_{i=1}^{2n+2},1)$. If there exist a constant $\lambda$ and two indices $i,j$ such that $\dfrac{(h,H_i)}{(h,H_j)}=\lambda\dfrac{(g,H_i)}{(g,H_j)}$ then $\lambda =1$. 
\end{lemma}
{\textit{Proof.}\ By Lemma \ref{3.1}, we see that $h$ and $g$ are linearly nondegenerate and have the characteristic functions of the same order with the characteristic function of $f$. Setting $ H=  \dfrac{(h,H_{i})}{(h,H_j)}\text{ and }G=\dfrac{(g,H_{i})}{(g,H_j)}$ and
\begin{align*}
&S_t'=\{z;\nu_{(h,H_t),>k_t}(z)>0\}\cup \{z;\nu_{(g,H_t),>k_t}(z)>0\}\quad (1\leqslant t\leqslant 2n+2) . 
\end{align*}
Then $H=\lambda G$. Supposing that $\lambda\ne 1$, since $H=G$ on the set $\bigcup_{t\ne i,j}T_t\setminus (S_i'\cup S_j')$, we have $\bigcup_{t\ne i,j}T_t\subset S_i'\cup S_j'$. Thus
\begin{align*}
0\ge&\sum_{t\ne i,j}N^{(1)}_{(f,H_t),\leqslant k_t}(r)-(N(r,S_i')+N(r,S_j'))\\
\ge&\dfrac{1}{2}\sum_{t\ne i,j}(N^{(1)}_{(h,H_t),\leqslant k_t}(r)+N^{(1)}_{(g,H_t),\leqslant k_t}(r))-(N(r,S_i')+N(r,S_j'))\\
\ge&\dfrac{1}{2n}\sum_{t\ne i,j}(N^{(n)}_{(h,H_t)}(r)+N^{(n)}_{(g,H_t)}(r))-\sum_{t=1}^{2n+2}(N^{(1)}_{(h,H_t),>k_t}(r)+N^{(1)}_{(g,H_t),>k_t}(r))\\
\ge&\dfrac{n-1}{2n}(T_h(r)+T_g(r))-\sum_{t=1}^{2n+2}\dfrac{1}{k_t+1}(T_h(r)+T_g(r))+o(T_f(r)).
\end{align*}
Letting $r\longrightarrow +\infty$, we get
$$\dfrac{n-1}{2n}\leqslant\sum_{t=1}^{2n+2}\dfrac{1}{k_t+1}.$$
This is a contradiction. Therefore $\lambda =1$. The lemma is proved \hfill$\square$
\begin{lemma}\label{4.2} Let $f^1,f^2,f^3$ be three elements of $\mathcal F(f,\{H_i,k_i\}_{i=1}^{2n+2},1)$, where $k_i\  (1\leqslant i\leqslant 2n+2)$ are positive integers or $+\infty$.
Suppose that $f^1\wedge f^2\wedge f^3\equiv 0$ and $V_i\sim V_j$ for some distinct indices $i$ and $j$. Then $f^1,f^2,f^2$ are not distinct. 
\end{lemma}
\textit{Proof.}\quad Suppose $f^1,f^2,f^2$ are distinct. Since $V_i\sim V_j$, we may suppose that $F_1^{ij}=F_2^{ij}\ne F_3^{ij}$. Since $f^1\wedge f^2\wedge f^3\equiv 0$ and $f^1\ne f^2$, there exists a meromorphic function $\alpha$ such that 
$$ F_3^{tj}=\alpha F_1^{tj}+(1-\alpha)F_2^{tj}\ (1\leqslant t\leqslant 2n+2). $$
This implies that $F_3^{ij}=F_1^{ij}=F^{ij}_2$. This is a contradiction. Hence $f^1,f^2,f^3$ are not distinct. The lemma is proved \hfill$\square$

\begin{lemma}\label{4.4} 
With the assumption of Theorem \ref{1.2}. Let $f^1,f^2,f^3$ be three maps in $\mathcal F(f,\{H_i,k_i\}_{i=1}^{2n+2},1)$. Suppose that $f^1,f^2,f^3$ are distinct and there are two indices $i, j \in \{1,2,\ldots ,2n+2\} \ (i\ne j)$ such that $V_i\not\cong V_j$ and
$$\Phi_{ij}^{\alpha}:= \Phi^{\alpha}(F_1^{ij},F_2^{ij},F_3^{ij})\equiv 0$$
for every $\alpha =(\alpha_1,\ldots ,\alpha_m)\in\Z^m_+$ with $|\alpha|=1$. Then for every $t\in\{1,\ldots ,2n+2\}\setminus \{i\}$, the following assertion hold:
\begin{itemize}
\item[(i)] $\Phi^{\alpha}_{it}\equiv 0$ for all $|\alpha| \leqslant 1,$
\item[(ii)] if $V_i\not\cong V_t$ then  $F_1^{ti},F_2^{ti},F_3^{ti}$ are distinct and
\begin{align*}
N^{(1)}_{(f,H_i),\leqslant k_i}(r)&\geqslant\sum_{s\ne i,t}N^{(1)}_{(f,H_s),\leqslant k_s}(r)-N^{(1)}_{(f,H_t),\leqslant k_t}(r)-2(N(r,S_i)+N(r,S_t))\\
&\geqslant \sum_{s\ne i,t}N^{(1)}_{(f,H_s),\leqslant k_s}(r)-N^{(1)}_{(f,H_t),\leqslant k_t}(r)-2\sum_{u=1}^3\sum_{s=i,t}N_{(f^u,H_s),\leqslant k_s}(r).
\end{align*} 
\end{itemize}
\end{lemma}
 \textit{Proof.} By the supposition $V_i\not\cong V_j$, we may assume that $F_2^{ji}-F_1^{ji}\ne 0$. 

(a) For all $\alpha\in \Z^m_+$ with $|\alpha|=1$, we have $\Phi_{ij}^{\alpha}=0$, and hence
\begin{align*}
\mathcal {D}^{\alpha}\biggl(\dfrac {F_3^{ji}-F_1^{ji}}{F_2^{ji}-F_1^{ji}}\biggl)=&\dfrac{1}{(F_2^{ji}-F_1^{ji})^2}\cdot\biggl ((F_2^{ji}-F_1^{ji})\cdot \mathcal {D}^{\alpha}(F_3^{ji}-F_1^{ji})\\
&\hspace{90pt}-(F_3^{ji}-F_1^{ji})\cdot \mathcal {D}^{\alpha}(F_2^{ji}-F_1^{ji})\biggl)\\
=&\dfrac{1}{{(F_2^{ji}-F_1^{ji})^2}}\cdot\left | 
\begin {array}{cccc}
1&1&1\\
F_1^{ji}&F_2^{ji} &F_3^{ji}\\
\mathcal {D}^{\alpha}(F_1^{ji}) &\mathcal {D}^{\alpha}(F_2^{ji}) &\mathcal {D}^{\alpha}(F_3^{ji})
\end {array}
\right| =0.
\end{align*}
Since the above equality hold for all $|\alpha|=1$, then there exists a constant $c\in \C$ such that
\begin{align*}
\dfrac {F_3^{ji}-F_1^{ji}}{F_2^{ji}-F_1^{ji}}=c
\end{align*}
By Theorem \ref{1.1}, we have $f^1\wedge f^2\wedge f^3=0.$ Then for each index $t\in\{1,\ldots ,2n+2\}\setminus\{i,j\}$ we have
\begin{align*}
0&= \det\left (
\begin{array}{ccc}
(f_1,H_i)&(f_1,H_j)&(f_1,H_{t})\\ 
(f_2,H_i)&(f_2,H_j)&(f_2,H_{t})\\
(f_3,H_i)&(f_3,H_j)&(f_3,H_{t})
\end{array}\right )
=\prod_{u=1}^3(f^u,H_i)\cdot\det\left (
\begin{array}{ccc}
1&F_1^{ji}&F_1^{ti}\\
1&F_2^{ji}&F_2^{ti}\\
1&F_3^{ji}&F_3^{ti}\\
\end{array}\right )\\
&=\prod_{u=1}^3(f^u,H_i)\cdot\det\left (
\begin{array}{ccc} 
F_2^{ji}-F_1^{ji}&F_2^{ti}-F_1^{ti}\\
F_3^{ji}-F_1^{ji}&F_3^{ti}-F_1^{ti}\\
\end{array}\right ).
\end{align*}
Thus
$$ (F_2^{ji}-F_1^{ji})\cdot (F_3^{ti}-F_1^{ti})= (F_3^{ji}-F_1^{ji})\cdot (F_2^{ti}-F_1^{ti}).$$
If $F_2^{ti}-F_1^{ti}=0$ then $F_3^{ti}-F_1^{ti}=0$, and hence $\Phi^{\alpha}_{it}=0$ for all $\alpha\in\Z^m_+$ with $|\alpha|<1$. Otherwise, we have
$$ \dfrac{F_3^{ti}-F_1^{ti}}{F_2^{ti}-F_1^{ti}}=\dfrac{F_3^{ji}-F_1^{ji}}{F_2^{ji}-F_1^{ji}}=c.$$
This also implies that
\begin{align*}
\Phi^{\alpha}_{it}&=F_1^{it}\cdot F_2^{it} \cdot F_3^{it}\cdot
\left |
\begin{array}{ccc}
1&1&1\\
F_1^{ti}&F_2^{ti} &F_3^{ti}\\
\mathcal {D}^{\alpha}(F_1^{ti}) &\mathcal {D}^{\alpha}(F_2^{ti}) &\mathcal {D}^{\alpha}(F_3^{ti})\\
\end {array}
\right|\\
&=F_1^{it}\cdot F_2^{it} \cdot F_3^{it}\cdot
\left |
\begin{array}{cc}
F_2^{ti}-F_1^{ti} &F_3^{ti}-F_1^{ti}\\
\mathcal {D}^{\alpha}(F_2^{ti}-F_1^{ti})  &\mathcal {D}^{\alpha}(F_3^{ti}-F_1^{ti})\\
\end {array}
\right|\\
&=F_1^{it}\cdot F_2^{it} \cdot F_3^{it}\cdot
\left |
\begin{array}{cc}
F_2^{ti}-F_1^{ti} &c(F_2^{ti}-F_1^{ti})\\
\mathcal {D}^{\alpha}(F_2^{ti}-F_1^{ti})  &c\mathcal {D}^{\alpha}(F_2^{ti}-F_1^{ti})
\end {array}
\right|=0.
\end{align*}
Then one always has $\Phi^{\alpha}_{it}=0$ for all $t\in\{1,\ldots ,2n+2\}\setminus \{i\}$. The first assertion is proved.

(b) We suppose that $V_i\not\cong V_t$. From the above part, we have
$$ c F_2^{si}+(1-c)F_1^{si}=F_3^{si}\ (s\ne i).$$
By the supposition $f^1,f^2,f^3$ are distinct, we have $c\not\in\{0,1\}$. This implies that $F_1^{ti},F_2^{ti},F_3^{ti}$ are distinct. 

We see that the second inequality is clear, then we prove the remain first inequality. We consider the meromorphic mapping $F^t$ of $\C^m $ into $ \P^1(\C)$ with a reduced representation 
$$F^t=(F_1^{ti}h_t:F_2^{ti}h_t),$$
where $h_t$ is a meromorphic function on $\C^m$. We see that 
\begin{align*}
T_{F^t}(r)=&T\biggl(r,\dfrac {F_1^{ti}}{F_2^{ti}}\biggl) \leqslant T(r,F_1^{ti})+T\biggl(r,\dfrac {1}{F_2^{ti}}\biggl)+O(1)\\ 
& \leqslant T(r,F_1^{ti})+T(r,F_2^{ti})+O(1)\leqslant T_{f^1}(r)+T_{f^2}(r)+O(1)=O(T_f(r)).
\end{align*}

For a point $z\not\in I(F^t)\cup S_i\cup S_t$ which is a zero of some functions $F_u^{ti}h_t\ (1\leqslant u\leqslant 3)$, then $z$ must be either zero of $(f,H_i)$ with multiplicity at most $k_i$ or zero of $(f,H_t)$ with multiplicity at most $k_t$, and hence
$$ \sum_{u=1}^3\nu^{(1)}_{F_u^{ti}h_t}(z)=1\leqslant \nu^{(1)}_{(f,H_i),\leqslant k_i}(z)+\nu^{(1)}_{(f,H_t),\leqslant k_t}(z). $$
This implies that
$$ \sum_{u=1}^3\nu^{(1)}_{F_u^{ti}h_t}(z)\leqslant \nu^{(1)}_{(f,H_i),\leqslant k_i}(z)+\nu^{(1)}_{(f,H_t),\leqslant k_t}(z)+\chi_{S_i}(z)+\chi_{S_t}(z) $$
outside an analytic subset of codimension two. By integrating both sides of this inequality, we get
\begin{align}\label{4.5}
\sum_{u=1}^3N^{(1)}_{F_u^{ti}h_t}(r)\leqslant N^{(1)}_{(f,H_i),\leqslant k_i}(r)+N^{(1)}_{(f,H_t),\leqslant k_t}(r)+N(r,S_i)+N(r,S_t).
\end{align}

By the second main theorem, we also have
\begin{align}\label{4.6}
||\ T_{F^t}(r)\leqslant\sum_{u=1}^3N^{(1)}_{F_u^{ti}h_t}(r)+o(T(r)).
\end{align}

On the other hand, applying the first main theorem to the map $F^t$ and the hyperplane  $\{w_0-w_1=0\}$ in $\P^1(\C),$ we have
\begin{align}\label{4.7}
T_{F^t}(r)&\geqslant N_{(F_1^{ti}-F_2^{ti})h_t}(r)\geqslant \sum_{\underset{v\ne i,t}{v=1}}^{2n+2}N^{(1)}_{(f,H_v),\leqslant k_v}(r)-N(r,S_i)-N(r,S_t).
\end{align}
Therefore, from (\ref{4.5}), (\ref{4.6}) and (\ref{4.7}) we have
\begin{align*}
||\ N^{(1)}_{(f,H_i),\leqslant k_i}(r)\geqslant \sum_{\underset{v\ne i,t}{v=1}}^{2n+2}N^{(1)}_{(f,H_v),\leqslant k_v}(r)-N^{(1)}_{(f,H_t),\leqslant k_t}(r)-2(N(r,S_i)+N(r,S_t))+o(T(r)).
\end{align*}
The second assertion of the lemma is proved.\hfill$\square$

\begin{lemma}\label{4.8}
With the assumption of Theorem \ref{1.2}, let $f^1,f^2,f^3$ be three meromorphic mappings in $\mathcal F(f,\{H_i,k_i\}_{i=1}^{2n+2},1)$.
Assume that there exist $i, j \in \{1,2,\ldots ,2n+2\} \ (i\ne j)$ and  $\alpha\in\Z^m_+$ with $|\alpha|=1$ such that
$\Phi^{\alpha}_{ij}\not\equiv 0.$
Then we have
\begin{align*}
T(r)&\geqslant \sum_{u=1}^{3}N^{(n)}_{(f^u,H_i),\leqslant k_i}(r)+\sum_{k=1}^3 N^{(n)}_{(f^k,H_j),\leqslant k_j}(r)+2\sum_{\overset{t=1}{t\ne i,j}}^{2n+2}N_{(f,H_t),\leqslant k_t}^{(1)}(r) \\
&\ \ \ -(2n+1)N^{(1)}_{(f,H_i),\leqslant k_i}(r)-(n+1)N^{(1)}_{(f,H_j),\leqslant k_j}(r)+N(r,\nu_j)\\
&\ \ \  -N(r,S_i)-N(r,S_j)-(2n-2)N(r,R_i)-(n-1)N(r,R_j)+o(T(r))\\
&\geqslant \sum_{u=1}^{3}N^{(n)}_{(f^u,H_i),\leqslant k_i}(r)+\sum_{k=1}^3 N^{(n)}_{(f^k,H_j),\leqslant k_j}(r)+2\sum_{\overset{t=1}{t\ne i,j}}^{2n+2}N_{(f,H_t),\leqslant k_t}^{(1)}(r) \\
&\ \ \ -(2n+1)N^{(1)}_{(f,H_i),\leqslant k_i}(r)-(n+1)N^{(1)}_{(f,H_j),\leqslant k_j}(r)+N(r,\nu_j)\\
&\ \ \ \ -\sum_{u=1}^3\bigl ((1+\dfrac{n-1}{3})N^{(1)}_{(f^u,H_j),> k_j}-(1+\dfrac{2n-2}{3})N_{(f^u,H_i),> k_i}\bigl )+o(T(r)).
\end{align*}
\end{lemma}
\textit{Proof.} The second inequality is clear. We remain prove the first inequality. We have
\begin{align*} \Phi^{\alpha}
&=F_1^{ij}\cdot F_2^{ij}\cdot F_3^{ij}\cdot 
\left | 
\begin {array}{cccc}
1&1&1\\
F_1^{ji}&F_2^{ji} &F_3^{ji}\\
\mathcal {D}^{\alpha}(F_1^{ji}) &\mathcal {D}^{\alpha}(F_2^{ji}) &\mathcal {D}^{\alpha}(F_3^{ji})\\
\end {array}
\right|\\
&=\left | 
\begin {array}{cccc}
F_1^{ij}&F_2^{ij} &F_3^{ij}\\
1&1&1\\
F_1^{ij}\mathcal {D}^{\alpha}(F_2^{ji}) &F_2^{ij}\mathcal {D}^{\alpha}(f^{ji}) &F_3^{ij}\mathcal {D}^{\alpha}(g^{ji})
\end {array}
\right|
\end{align*}
Thus
\begin{align}\label{4.9}
\begin{split}
\Phi_{ij}^{\alpha}&=F_1^{ij}\biggl(\dfrac{\mathcal {D}^{\alpha}(F_3^{ji})}{F^{ji}_3}-\dfrac{\mathcal {D}^{\alpha}(F_2^{ji})}{F^{ji}_2}\biggl)
+F^{ij}_2\biggl(\dfrac{\mathcal {D}^{\alpha}(F_1^{ji})}{F^{ji}_1}-\dfrac{\mathcal {D}^{\alpha}(F_3^{ji})}{F^{ji}_3}\biggl)\\
&\ \ \ +F^{ij}_3\biggl(\dfrac{\mathcal {D}^{\alpha}(F_2^{ji})}{F^{ji}_2}-\dfrac{\mathcal {D}^{\alpha}(F_1^{ji})}{F^{ji}_1}\biggl).
\end{split}
\end{align}
By the Logarithmic Derivative Lemma, it follows that
\begin{align*}
m(r,\Phi^\alpha_{ij})\leqslant\sum_{u=1}^3m(r,F_u^{ij})+2\sum_{u=1}^{3}m\biggl(\dfrac{\mathcal{ D}^{\alpha}(F_u^{ji})}{F_v^{ji}}\biggl)+O(1)
\leqslant\sum_{u=1}^3m(r,F_u^{ij})+o(T_f(r)).
\end{align*}
Therefore, we have
\begin{align*}
T(r)&\geqslant\sum_{u=1}^3T(r,F_u^{ij})=\sum_{u=1}^3(m(r,F_u^{ij})+N_{\frac{1}{F_u^{ij}}}(r))=m(r,\Phi^\alpha_{ij})+\sum_{u=1}^3N_{\frac{1}{F_u^{ij}}}(r)+o(T(r))\\
&\geqslant T(r,\Phi^\alpha_{ij})-N_{\frac{1}{\Phi^\alpha_{ij}}}+\sum_{u=1}^3N_{\frac{1}{F_u^{ij}}}(r)+o(T(r))\\
&\geqslant N_{\Phi^\alpha_{ij}}(r)-N_{\frac{1}{\Phi^\alpha_{ij}}}+\sum_{u=1}^3N_{\frac{1}{F_u^{ij}}}(r)+o(T(r))\\
&= N(r,\nu_{\Phi^\alpha_{ij}})+\sum_{u=1}^3N_{\frac{1}{F_u^{ij}}}(r)+o(T(r)).
\end{align*}
Then, in order to prove the lemma, it is sufficient for us to prove 
\begin{align}\notag
N(r,\nu_{\Phi^\alpha_{ij}})&\geqslant \sum_{u=1}^{3}N^{(n)}_{(f^u,H_i),\leqslant k_i}(r)+\sum_{k=1}^3 N^{(n)}_{(f^k,H_j),\leqslant k_j}(r)+2\sum_{\overset{t=1}{t\ne i,j}}^{2n+2}N_{(f,H_t),\leqslant k_t}^{(1)}(r) \\
\notag
&\ \ \ -(2n+1)N^{(1)}_{(f,H_i),\leqslant k_i}(r)-(n+1)N^{(1)}_{(f,H_j),\leqslant k_j}(r)-\sum_{u=1}^3N_{\frac{1}{F_u^{ij}}}(r)+N(r,\nu_j)\\
\label{4.10}
&\ \ \  -N(r,S_i)-N(r,S_j)-(2n-2)N(r,R_i)-(n-1)N(r,R_j)+o(T(r)).
\end{align}

Denote by $S$ the set of all singularities of $f^{-1}(H_t)\ (1\leqslant t\leqslant q)$. Then $S$ is an analytic subset of codimension at least two in $\C^m$. We set 
$$I=S\cup \bigcup_{s\ne t}\{z;\nu_{(f,H_s),\leqslant k_s}(z)\cdot\nu_{(f,H_t),\leqslant k_t}(z)>0\}.$$
Then $I$ is also an analytic subset of codimension at least two in $\C^m$. 

In order to prove the inequality (\ref{4.10}), it is sufficient for us to show that the following inequality 
\begin{align}\label{4.11}
P:\overset{Def}= &\sum_{u=1}^{3}\nu^{(n)}_{(f^u,H_i),\leqslant k_i}+\sum_{u=1}^3\nu^{(n)}_{(f^k,H_j),\leqslant k_j}+2\sum_{\overset{t=1}{t\ne i,j}}^{2n+2}\chi_{T_t}-(2n+1)\chi_{T_i}-(n+1)\chi_{T_j}\\
\notag
&-\sum_{u=1}^3\nu^\infty_{F_u^{ij}}+\chi_{\nu_j}-\chi_{S_i}-\chi_{S_j}-2(n-1)\chi_{R_i}-(n-1)\chi_{R_j}\leqslant \nu_{\Phi^\alpha_{ij}}.
\end{align}
holds outside the set $I$. 

Indeed, for $z\not\in I$, we distinguish the following cases:

\vskip0.2cm
\noindent
\textit{Case 1:}  $z\in T_t\setminus S_i\cup S_j \ (t\ne i,j)$. We see that $P(z)=2$.  We write $\Phi^\alpha_{ij}$ in the form
$$ \Phi^{\alpha}_{ij}
=F_1^{ij}\cdot F_2^{ij}\cdot F_3^{ij}
\times\left | 
\begin {array}{cccc}
 \bigl (F_1^{ji}-F_{2}^{ji}\bigl ) & \bigl (F_1^{ji}-F_{3}^{ji}\bigl )\\
\mathcal{D}^{\alpha}\bigl ( F_1^{ji}-F_{2}^{ji}\bigl ) & \mathcal{D}^{\alpha}\bigl ( F_1^{ji}-F_{3}^{ji}\bigl )
\end {array}
\right|. $$
Then  by the assumption that $f^1,f^2,f^3$ are identify on $T_t$, we have $F_1^{ji}=F_2^{ji}=F_3^{ji}$ on $T_t\setminus S_i$. The property of the wronskian implies that $\nu_{\Phi^{\alpha}_{ij}}(z)\geqslant 2=P(z)$.
 
\vskip0.2cm
\noindent
\textit{Case 2:} $z\in T_t\cap (S_i\cup S_j)\ (t\ne i,j)$. We see that $P(z)\leqslant -\sum_{u=1}^3\nu^\infty_{F_u^{ij}}(z)-1.$ 

From (\ref{4.9}) we see that 
$$ \nu_{\Phi^{\alpha}_{ij}}(z)\geqslant\min\{\nu_{F_1^{ij}}(z)-1,\nu_{F_2^{ij}}(z)-1,\nu_{F_3^{ij}}(z)-1\}\geqslant P(z). $$

\vskip0.2cm
\noindent
\textit{Case 3:} $z\in T_i\setminus S_j$. We have 
$$P(z)= \sum_{u=1}^3\nu^{(n)}_{(f^u,H_i),\leqslant k_i}(z)-(2n+1)\leqslant\min_{1\leqslant u\leqslant 3}\{\nu^{(n)}_{(f^u,H_i),\leqslant k_i}(z)\}-1.$$
We may assume that $\nu_{(f^1,H_{i})}(z)\leqslant\nu_{(f^2,H_{i})}(z)\leqslant\nu_{(f^3,H_{i})}(z)$. We write
\begin{align*}
\Phi^{\alpha}_{ij}
=F_1^{ij}\biggl [F_2^{ij}(F_1^{ji}-F_{2}^{ji})F_3^{ij}\mathcal{D}^{\alpha} (F_1^{ji}-F_{3}^{ji})-F_3^{ij}(F_1^{ji}-F_{2}^{ji})F_2^{ij}\mathcal{D}^{\alpha} (F_1^{ji}-F_{2}^{ji})\biggl ]
\end{align*}
It is easy to see that $F_2^{ij}(F_1^{ji}-F_{2}^{ji})$ and $F_3^{ij}(F_1^{ji}-F_{3}^{ji})$ are holomorphic on a neighborhood of $z$ and 
$$ \nu^{\infty}_{F_3^{ij}\mathcal{D}^{\alpha}(F_1^{ji}-F_{3}^{ji})}(z)\leqslant 1 $$
and 
$$ \nu^{\infty}_{F_2^{ij}\mathcal{D}^{\alpha} (F_1^{ji}-F_{2}^{ji})}(z)\leqslant 1. $$
Therefore, it implies that
\begin{align*}
\nu_{\Phi^{\alpha}_{ij}}(z)&\geqslant \nu^{(n)}_{(f^1,H_i),\leqslant k_i}(z)-1\geqslant P(z).
\end{align*}

\vskip0.2cm
\noindent
\textit{Case 4:} $z\in T_i\cap S_j$. The assumption that $f^1,f^2,f^3$ are identity on $T_i$ yields that $z\in R_j$. We have
$$P(z)\leqslant \sum_{u=1}^3\nu_{(f^u,H_i),\leqslant k_i}^{(n)}(z)-\sum_{u=1}^3\nu^\infty_{F_u^{ij}}(z)-(2n+1)-n\leqslant -\sum_{u=1}^3\nu^\infty_{F_u^{ij}}(z)-1.$$
We have 
\begin{align*}
\nu_{\Phi^{\alpha}_{ij}}(z)&\geqslant\min\{\nu_{F_1^{ij}}(z)-1,\nu_{F_2^{ij}}(z)-1,\nu_{F_3^{ij}}(z)-1\}\geqslant -\sum_{u=1}^3\nu^\infty_{F_u^{ij}}(z)-1\geqslant P(z).
\end{align*}

\vskip0.2cm
\noindent
\textit{Case 5:} $z\in T_j$.  We may assume that 
$$\nu_{F_1^{ji}}(z)=d_1\geqslant \nu_{F_2^{ji}}(z)=d_2\geqslant\nu_{F_3^{ji}}(z)=d_3.$$
 Choose a holomorphic function $h$ on $\C^m$ with multiplicity $1$ at $z$ such that
$F_u^{ji}=h^{d_u}\varphi_u\ (1\leqslant u\leqslant 3),$ where $\varphi_u$ are meromorphic on $\C^m$ and 
holomorphic on a neighborhood of $z$. Then
\begin{align*}
 \Phi^{\alpha}_{ij}&=F_1^{ij}\cdot F_2^{ij}\cdot F_3^{ij}\cdot 
\left | 
\begin {array}{ccc}
F_2^{ji}-F_1^{ji}&F_3^{ji}-F_1^{ji}\\
\mathcal {D}^{\alpha}(F_2^{ji}-F_1^{ji}) &\mathcal {D}^{\alpha}(F_3^{ji}-F_1^{ji})\\
\end {array}
\right|\\
&= F_1^{ij}\cdot F_2^{ij}\cdot F_3^{ij}\cdot h^{d_2+d_3}\cdot 
\left | 
\begin {array}{ccc}
\varphi_2-h^{d_1-d_2}\varphi_1&\varphi_3-h^{d_1-d_3}\varphi_1\\
\dfrac{\mathcal {D}^{\alpha}(h^{d_2-d_3}\varphi_2-h^{d_1-d_3}\varphi_1)}{h^{d_2-d_3}} &\mathcal {D}^{\alpha}(\varphi_3-h^{d_1-d_3}\varphi_1)\\
\end {array}
\right|.
\end{align*}
This yields that
\begin{align*}
\nu_{\Phi^\alpha_{ij}}(z)&\geqslant\sum_{u=1}^3 \nu_{F_u^{ij}}(z)+d_2+d_3-\max\{0,\min\{1,d_2-d_3\}\}.
\end{align*}
If $z\not\in S_i$ then 
\begin{align*}
P(z)&=-\sum_{u=1}^3 \nu^{\infty}_{F_u^{ij}}(z)+\sum_{u=1}^3\min\{n,d_u\}-(n+1)+\chi_{\nu_j},
\end{align*}
 and
\begin{align*}
\nu_{\Phi^\alpha_{ij}}(z)&\geqslant -\sum_{u=1}^3 \nu^{\infty}_{F_u^{ij}}(z)+\sum_{u=1}^3\nu^0_{F_u^{ij}}(z)+d_2+d_3-1+\chi_{\nu_j}\\
&\geqslant -\sum_{u=1}^3 \nu^{\infty}_{F_u^{ij}}(z)+d_2+d_3-1+\chi_{\nu_j}\geqslant P(z).
\end{align*}
Otherwise, if $z\in S_i$ then $z\in R_i$, and hence 
$$P(z)\leqslant  \sum_{u=1}^3\nu^{(n)}_{(f^u,H_j),\leqslant k_j}-\sum_{u=1}^3 \nu^{\infty}_{F_u^{ij}}(z)-3n-1+\chi_{\nu_j}\leqslant -\sum_{u=1}^3 \nu^{\infty}_{F_u^{ij}}(z)-3n,$$
\begin{align*}
\text{and \ \ }\nu_{\Phi^\alpha_{ij}}(z)&\geqslant -\sum_{u=1}^3 \nu^{\infty}_{F_u^{ij}}(z)+\sum_{u=1}^3\nu^0_{F_u^{ij}}(z)+d_2+d_3-1\\
&\geqslant -\sum_{u=1}^3 \nu^{\infty}_{F_u^{ij}}(z)+\max\{0,-d_1\}+\max\{d_2,0\}+\max\{d_3,0\}-1\geqslant P(z).
\end{align*}

\vskip0.2cm
\noindent
\textit{Case 6:} $z\in (S_i\cup S_j)\setminus (\bigcup_{t=1}^{2n+2} T_t)$. Similarly as Case 5, we have
\begin{align*}
\nu_{\Phi^\alpha_{ij}}(z)&\geqslant -\sum_{u=1}^3 \nu^{\infty}_{F_u^{ij}}(z)+\max\{0,-d_1\}+\max\{d_2,0\}+\max\{d_3,0\}-1\\
&\ge-\sum_{u=1}^3 \nu^{\infty}_{F_u^{ij}}(z)-1\geqslant -\sum_{u=1}^3 \nu^{\infty}_{F_u^{ij}}(z)-\chi_{S_i}-\chi_{S_j}\geqslant P(z).
\end{align*}

From the above six cases, we see that the inequality (\ref{4.11}) holds. Hence the lemma is proved \hfill$\square$

\begin{proof}[ Proof of theorem \ref{1.2}]  Suppose that there exits three distinct meromorphic mappings $f^1,f^2,f^3$ in $\mathcal F(f,\{H_i,k_i\}_{i=1}^{2n+2},1)$.  By Lemma \ref{1.1}, we have $f^1\wedge f^2\wedge f^3\equiv 0$. Without loss of generality, we may assume that

$\underbrace{V_1\cong \cdots\cong V_{l_1}}_{\text { group } 1}\not\cong
\underbrace{V_{l_1+1}\cong\cdots\cong V_{l_2}}_{\text { group } 2}\not\equiv \underbrace{V_{l_2+1}\cong \cdots\cong V_{l_3}}_{\text { group } 3}\not\cong\cdots \not\cong\underbrace{V_{l_s+1}\cong\cdots\cong V_{l_{s+1}}}_{\text { group } s},$\\
where $l_s=2n+2.$ 

Denote by $P$ the set of all $i\in\{1,\ldots ,2n+2\}$ satisfying there exist $j\in\{1,\ldots ,2n+2\}\setminus\{i\}$ such that $V_i\not\cong V_j$ and $\Phi^{\alpha}_{ij}\equiv 0$ for all $\alpha\in\Z^m_+$ with $|\alpha|\leqslant 1.$ We consider the following three cases.

\vskip0.2cm
\noindent
\textit{Case 1:} $\sharp P\geqslant 2$. Then $P$ contains two elements $i,j$. Then we have $\Phi^{\alpha}_{ij}=\Phi^{\alpha}_{ji}=0$ for all $\alpha\in\Z^m_+$ with $|\alpha|\leqslant 1.$ By Lemma \ref{4.3}, there exist two functions, for instance they are $F_1^{ij}$ and $F^2_{ij}$, and a constant $\lambda$ such that $F_1^{ij}=\lambda F_2^{ij}$. This yields that $F_1^{ij}=F^{ij}_2$ (by Lemma \ref{4.1}). Then by Lemma \ref{4.4} (ii), we easily see that $V_i\cong V_j$, i.e., $V_i$ and $V_j$ belong to the same group in the above partition. 

Without loss of generality, we may assum that $i=1$ and $j=2$. Since $f^1,f^2,f^3$ are supposed to be distinct, the number of each group in the above partition is less than $n+1$. Hence we have $V_1\cong V_2\not\cong V_t$ for all $t\in\{n+1,\ldots ,2n+2\}$. Then by Lemma \ref{4.4} (ii), we have
\begin{align*}
 N^{(1)}_{(f,H_i),\leqslant k_1}(r)+N^{(1)}_{(f,H_t),\leqslant k_t}(r)&\geqslant\sum_{s\ne 1,t}N^{(1)}_{(f,H_s),\leqslant k_s}(r)-2\sum_{u=1}^3\sum_{s=1,t}N^{(1)}_{(f^u,H_s),>k_s}(r),\\
 \text{and }N^{(1)}_{(f,H_2),\leqslant k_2}(r)+N^{(1)}_{(f,H_t),\leqslant k_t}(r)&\geqslant\sum_{s\ne 2,t}N^{(1)}_{(f,H_s),\leqslant k_s}(r)-2\sum_{u=1}^3\sum_{s=2,t}N^{(1)}_{(f^u,H_s),>k_s}(r).
\end{align*}
Summing-up both sides of the above two inequalities, we get
\begin{align*}
2N^{(1)}_{(f,H_t),\leqslant k_t}(r)\ge&2\sum_{s\ne 1,2,t}N^{(1)}_{(f,H_s),\leqslant k_s}(r)-2\sum_{u=1}^3(N^{(1)}_{(f^u,H_1),>k_1}(r)+N^{(1)}_{(f^u,H_2),>k_2}(r)\\
&+2N^{(1)}_{(f^u,H_t),>k_t}(r)).
\end{align*}
After summing-up both sides of the above inequalities over all $t\in\{n+1,2n+2\}$, we easily obtain
\begin{align*}
||\ \sum_{u=1}^3((n+2)&(N^{(1)}_{(f^u,H_1),>k_1}(r)+N^{(1)}_{(f^u,H_2),>k_2}(r))+2\sum_{t=n+1}^{2n+2}N^{(1)}_{(f^u,H_t),>k_t}(r))\\
&\geqslant (n+2)\sum_{t=3}^nN^{(1)}_{(f,H_t),\leqslant k_t}(r)+n\sum_{t=n+1}^{2n+2}N^{(1)}_{(f,H_t),\leqslant k_t}(r)\\
&\geqslant n\sum_{t=3}^{2n+2}N^{(1)}_{(f,H_t),\leqslant k_t}(r)\geqslant \dfrac{n}{3}\sum_{u=1}^3\sum_{t=3}^{2n+2}N^{(1)}_{(f_u,H_t),\leqslant k_t}(r)\\
&\geqslant\dfrac{n}{3}\sum_{u=1}^3\sum_{t=3}^{2n+2}N^{(1)}_{(f_u,H_t)}(r)-\dfrac{n}{3}\sum_{u=1}^3\sum_{t=3}^{2n+2}N^{(1)}_{(f^u,H_t),>k_t}(r)\\
&\geqslant\dfrac{1}{3}\sum_{u=1}^3\sum_{t=3}^{2n+2}N^{(n)}_{(f_u,H_t)}(r)-\dfrac{n}{3}\sum_{u=1}^3\sum_{t=3}^{2n+2}N^{(1)}_{(f^u,H_t),>k_t}(r)\\
&\geqslant\dfrac{n-1}{3}T(r)-\dfrac{n}{3}\sum_{u=1}^3\sum_{t=3}^{2n+2}N^{(1)}_{(f^u,H_t),>k_t}(r)+o(T(r)).
\end{align*}
Therefore, we have
\begin{align*}
\dfrac{n-1}{3}T(r)&\leqslant (n+2)\sum_{u=1}^3\sum_{t=1}^{2n+2}N^{(1)}_{(f^u,H_t),>k_t}(r)\leqslant (n+2)\sum_{u=1}^3\sum_{t=1}^{2n+2}\dfrac{1}{k_t+1}N_{(f^u,H_t),>k_t}(r)\\
&\leqslant (n+2)\sum_{t=1}^{2n+2}\dfrac{1}{k_t+1}T(r). 
\end{align*}
Letting $r\longrightarrow +\infty$, we get
$$  \dfrac{n-1}{3(n+2)}\leqslant \sum_{t=1}^{2n+2}\dfrac{1}{k_t+1}.$$
This is a contradiction.

\vskip0.2cm 
\noindent
\textbf{\textit{Case 2:}} $\sharp P=1$. We assume that $P=\{1\}$. We easily see that $V_{1}\not\cong V_i$ for all $i=2,\ldots ,2n+2$ (otherwise $i\in P$, this contradict to $\sharp P=1$). Then by Lemma \ref{4.4} (ii), we have
\begin{align*}
N^{(1)}_{(f,H_1),\leqslant k_1}(r)\geqslant\sum_{s\ne 1,i}N^{(1)}_{(f,H_s),\leqslant k_s}(r)-N^{(1)}_{(f,H_i),\leqslant k_i}(r)-2\sum_{u=1}^3\sum_{s=1,i}N^{(1)}_{(f^u,H_s),>k_s}(r)+o(T(r)).
\end{align*}
Summing-up both sides of the above inequality over all $i=2,\ldots ,2n+2$, we get
\begin{align}\label{4.12}
\begin{split}
(2n+1)N^{(1)}_{(f,H_1),\leqslant k_1}(r)\ge& (2n-1)\sum_{i=2}^{2n+2}N^{(1)}_{(f,H_i),\leqslant k_i}(r)-2\sum_{u=1}^3\sum_{i=2}^{2n+2}N^{(1)}_{(f^u,H_i),>k_s}(r)\\
&-2(2n+1)\sum_{u=1}^3N^{(1)}_{(f^u,H_1),>k_1}(r)+o(T(r)).
\end{split}
\end{align}
We also see that $i\not\in P$ for all $2\leqslant i\leqslant 2n+2$. We set
$$ \sigma (i)=\begin{cases}
i+n&\text{ if }i\leqslant n+2,\\
i-n&\text{ if } n+2<i\leqslant 2n+2.
\end{cases}$$
Then we easily see that $i$ and $\sigma (i)$ belong to two distinct groups, i.e, $V_i\not\cong V_{\sigma (i)}$, for all $i\in\{2,\ldots ,2n+2\}$, and hence $\Phi^{\alpha}_{i\sigma (i)}\not\equiv 0$ for all $\alpha\in\Z^m_+$ with $|\alpha|\leqslant 1$. By Lemma \ref{4.5} we have
\begin{align*}
T(r)\ge& \sum_{u=1}^{3}\sum_{t=i,\sigma (i)}N^{(n)}_{(f^u,H_t),\leqslant k_t}(r)-(2n+1)N^{(1)}_{(f,H_i),\leqslant k_i}(r)-(n+1)N^{(1)}_{(f,H_{\sigma (i)}),\leqslant k_{\sigma (i)}}(r)\\
&+2\sum_{\overset{t=1}{t\ne i,{\sigma (i)}}}^{2n+2}N_{(f,H_t),\leqslant k_t}^{(1)}(r)-\sum_{u=1}^{3}\biggl (\dfrac{2n+1}{3}N^{(1)}_{(f^u,H_{i}),>k_{i}}(r)+\dfrac{n+2}{3}N^{(1)}_{(f^u,H_{\sigma (i)}),>k_{\sigma (i)}}\biggl )\\
&+o(T(r)).
\end{align*}
Summing-up both sides of the above inequalities over all $i\in\{2,\ldots ,2n+2\}$, we get
\begin{align*}
(2n+1)&T(r)\geqslant 2\sum_{i=2}^{2n+2}\sum_{u=1}^{3}N^{(n)}_{(f^u,H_i),\leqslant k_i}(r)+(n-4)\sum_{i=2}^{2n+2}N^{(1)}_{(f,H_i),\leqslant k_i}(r)\\
&\ \ \ \ \ \ +2(2n+1)N^{(1)}_{(f,H_1),\leqslant k_1}(r)-(n+1)\sum_{u=1}^3\sum_{i=2}^{2n+2}N^{(1)}_{(f^u,H_{i}),>k_{i}}+o(T(r))\\
&\geqslant 2\sum_{i=2}^{2n+2}\sum_{u=1}^{3}N^{(n)}_{(f^u,H_i),\leqslant k_i}(r)+\dfrac{5n-6}{3}\sum_{u=1}^3\sum_{i=2}^{2n+2}N^{(1)}_{(f^u,H_i),\leqslant k_i}(r)\\
&-(8n+4)\sum_{u=1}^3N^{(1)}_{(f^u,H_1),> k_1}(r)-(n+5)\sum_{u=1}^3\sum_{i=2}^{2n+2}N^{(1)}_{(f^u,H_{i}),>k_{i}}+o(T(r))+o(T(r))\\
&\geqslant \dfrac{11n-6}{3n}\sum_{u=1}^{3}\sum_{i=1}^{2n+2}N^{(n)}_{(f^u,H_i),\leqslant k_i}(r)
\end{align*}

\begin{align*}
&-\dfrac{4n+2}{3}\sum_{u=1}^3N^{(1)}_{(f^u,H_1),> k_1}(r)-(n+1)\sum_{u=1}^3\sum_{i=2}^{2n+2}N^{(1)}_{(f^u,H_{i}),>k_{i}}+o(T(r))+o(T(r))\\
&\geqslant \dfrac{11n-6}{3n}\sum_{u=1}^{3}\sum_{i=2}^{2n+2}N^{(n)}_{(f^u,H_i)}(r)\\
&-(8n+4)\sum_{u=1}^3N^{(1)}_{(f^u,H_1),> k_1}(r)-\dfrac{14n+3}{3}\sum_{u=1}^3\sum_{i=2}^{2n+2}N^{(1)}_{(f^u,H_{i}),>k_{i}}+o(T(r))+o(T(r))\\
&\geqslant \dfrac{11n-6}{3}T(r)-(8n+4)\sum_{i=1}^{2n+2}\dfrac{1}{k_i+1}T(r)+o(T(r)).
\end{align*}
Letting $r\longrightarrow +\infty$, we get
\begin{align*}
\dfrac{5n-9}{24n+12}\leqslant \sum_{i=1}^{2n+2}\dfrac{1}{k_i+1}.
\end{align*}
This is a contradiction. 

\vskip0.2cm
\noindent
\textbf{\textit{Case 3:}} $P=\emptyset$. Then for all $i\ne j$, by Lemma \ref{4.5} we have
\begin{align*}
T(r)&\geqslant \sum_{u=1}^{3}N^{(n)}_{(f^u,H_i),\leqslant k_i}(r)+\sum_{k=1}^3 N^{(n)}_{(f^k,H_j),\leqslant k_j}(r)+2\sum_{\overset{t=1}{t\ne i,j}}^{2n+2}N_{(f,H_t),\leqslant k_t}^{(1)}(r) \\
&\ \ -(2n+1)N^{(1)}_{(f,H_i),\leqslant k_i}(r)-(n+1)N^{(1)}_{(f,H_j),\leqslant k_j}(r)+N(r,\nu_j)\\
&\ \ -\sum_{u=1}^3\bigl ((1+\dfrac{n-1}{3})N^{(1)}_{(f^u,H_j),> k_j}(r)+(1+\dfrac{2n-2}{3})N^{(1)}_{(f^u,H_i),>k_i}(r)\bigl )+o(T(r)).
\end{align*}
Summing-up both sides of the above inequalities over all pairs $(i,j)$ we get
\begin{align}\notag
(2n+2)T(r)\geqslant &2\sum_{u=1}^3\sum_{t=1}^{2n+2}N^{(n)}_{(f^u,H_t),\leqslant k_t}(r)+(n-2)\sum_{t=1}^{2n+2}N^{(1)}_{(f,H_t),\leqslant k_t}(r)+\sum_{t=1}^{2n+2}N(r,\nu_t)\\
& -(n+1)\sum_{u=1}^3\sum_{t=1}^{2n+2}N_{(f^u,H_i),> k_i}+o(T(r)).
\label{4.13}
\end{align}

On the other hand, by Lemma \ref{4.2}, we see that $V_j\not\sim V_l$ for all $j\ne l$. Hence, we have
$$P_{st}^{jl}:\overset{Def}=(f^s,H_j)(f^t,H_l)-(f^t,H_l)(f^s,H_j)\not\equiv 0\ (s\ne t, j\ne l).$$
\begin{claim}
With $i\ne j\ne l\ne i$, for every $z\in T_i$ we have 
$$\sum_{1\leqslant s<t\leqslant 3}\nu_{P_{st}^{jl}}(z)\geqslant 4\chi_{T_i}(z)-\chi_{\nu_i}(z).$$ 
\end{claim}
Indeed, for $z\in T_i\setminus \nu_i$, we may assume that $\nu_{(f^1,H_i)}(z)<\nu_{(f^2,H_i)}(z)\leqslant\nu_{(f^3,H_i)}(z)$. Since $f^1\wedge f^2\wedge f^3\equiv 0$, we have $\det (V_i,V_j,V_l)\equiv 0$, and hence
\begin{align*}
(f^1,H_i)P_{23}^{jl}=(f^2,H_i)P_{13}^{jl}-(f^3,H_i)P_{12}^{jl}.
\end{align*}
This yields that 
$$ \nu_{P_{23}^{jl}}(z)\geqslant 2 $$
and hence $\sum_{1\leqslant s<t\leqslant 3}\nu_{P_{st}^{jl}}(z)\geqslant 4=4\chi_{T_i}(z)-\chi_{\nu_i}(z)$ .

Now, for $z\in\nu_i$, we have $\sum_{1\leqslant s<t\leqslant 3}\nu_{P_{st}^{jl}}(z)\geqslant 3=4\chi_{T_i}(z)-\chi_{\nu_i}(z)$. Hence, the claim is proved.

\vskip0.2cm 
On the other hand, with $i=j$ or $i=l$, for every $z\in\{\nu_{(f,H_i),\leqslant k_i}(z)>0\}$ we see that 
\begin{align*}
\nu_{P_{st}^{jl}}(z)\ge&\min\{\nu_{(f^s,H_i),\leqslant k_i}(z),\nu_{(f^t,H_i),\leqslant k_i}(z)\}\\
\ge& \nu^{(n)}_{(f^s,H_i),\leqslant k_i}(z)+\nu^{(n)}_{(f^t,H_i),\leqslant k_i}(z)-n \nu^{(1)}_{(f,H_i),\leqslant k_i}(z).
\end{align*}
$$\text{and hence \ }\sum_{1\leqslant s<t\leqslant 3}\nu_{P_{st}^{jl}}(z)\geqslant 2\sum_{u=1}^3\nu^{(n)}_{(f^u,H_i),\leqslant k_i}(z)-3n\nu^{(1)}_{(f,H_i),\leqslant k_i}(z).$$
Combining this inequality and the above claim, we have
$$\sum_{1\leqslant s<t\leqslant 3}\nu_{P_{st}^{jl}}(z)\geqslant \sum_{i=j,l}(2\sum_{u=1}^3\nu^{(n)}_{(f^u,H_i),\leqslant k_i}(z)-3n\nu^{(1)}_{(f,H_i),\leqslant k_i}(z))+\sum_{i\ne j,l}(4\nu^{(1)}_{(f,H_i),\leqslant k_i}(z)-\chi_{\nu_i}(z)).$$
This yields that
\begin{align}\label{4.17}
\begin{split}
\sum_{1\leqslant s<t\leqslant 3}N_{P_{st}^{jl}}(z)\geqslant &\sum_{i=j,l}(2\sum_{u=1}^3N^{(n)}_{(f^u,H_i),\leqslant k_i}(r)-3nN^{(1)}_{(f,H_i),\leqslant k_i}(r))\\
&+\sum_{i\ne j,l}(4N^{(1)}_{(f,H_i),\leqslant k_i}(r)-N(r,\nu_i)).
\end{split}
\end{align}
On the other hand, be Jensen formula, we easily see that
$$ N_{P_{st}^{jl}}(z)\leqslant T_{f^s}(r)+T_{f^t}(r)+o(T(r))\ (1\leqslant s<t\leqslant 3). $$
Then the inequality (\ref{4.17}) implies that
\begin{align*}
2T(r)\geqslant &\sum_{i=j,l}(2\sum_{u=1}^3N^{(n)}_{(f^u,H_i),\leqslant k_i}(r)-3nN^{(1)}_{(f,H_i),\leqslant k_i}(r))+\sum_{i\ne j,l}(4N^{(1)}_{(f,H_i),\leqslant k_i}(r)-N(r,\nu_i)).
\end{align*}
Summing-up both sides of the above inequalities over all pair $(j,l)$, we obtain
\begin{align*}
2T(r)\ge& \dfrac{2}{n+1}\sum_{u=1}^3\sum_{i=1}^{2n+2}N^{(n)}_{(f^u,H_i),\leqslant k_i}(r)+\dfrac{n}{3\times (n+1)}\sum_{u=1}^3\sum_{i=1}^{2n+2}N^{(1)}_{(f^u,H_i),\leqslant k_i}(r)\\
&-\dfrac{n}{n+1}\sum_{i=1}^{2n+2}N(r,\nu_i)+o(T(r)).
\end{align*}
Thus
\begin{align*}
\sum_{i=1}^{2n+2}N(r,\nu_i)\ge&\dfrac{2}{n}\sum_{u=1}^3\sum_{i=1}^{2n+2}N^{(n)}_{(f^u,H_i),\leqslant k_i}(r)+\dfrac{1}{3}\sum_{u=1}^3\sum_{i=1}^{2n+2}N^{(1)}_{(f^u,H_i),\leqslant k_i}(r)\\
&-\dfrac{2(n+1)}{n}T(r)+o(T(r)).
\end{align*}
Using this estimate, from (\ref{4.13}) we have
\begin{align*}
(2n+2)T(r)\geqslant &(2+\dfrac{2}{n})\sum_{u=1}^3\sum_{t=1}^{2n+2}N^{(n)}_{(f^u,H_t),\leqslant k_t}(r)+\dfrac{n-1}{3}\sum_{u=1}^3\sum_{t=1}^{2n+2}N^{(1)}_{(f_u,H_t),\leqslant k_t}(r)\\
& -\dfrac{2(n+1)}{n}T(r)-(n+1)\sum_{u=1}^3\sum_{t=1}^{2n+2}N_{(f^u,H_i),> k_i}+o(T(r)).\\
&\geqslant (2+\dfrac{2}{n}+\dfrac{n-1}{3n})\sum_{u=1}^3\sum_{t=1}^{2n+2}N^{(n)}_{(f^u,H_t),\leqslant k_t}(r)-\dfrac{2(n+1)}{n}T(r)\\
& -(n+1)\sum_{u=1}^3\sum_{t=1}^{2n+2}N_{(f^u,H_i),> k_i}+o(T(r)).\\
&\geqslant (2+\dfrac{2}{n}+\dfrac{n-1}{3n})\sum_{u=1}^3\sum_{t=1}^{2n+2}N^{(n)}_{(f^u,H_t)}(r)-\dfrac{2(n+1)}{n}T(r)\\
&-(3n+3+\dfrac{n-1}{3})\sum_{u=1}^3\sum_{t=1}^{2n+2}N_{(f^u,H_i),> k_i}+o(T(r)).\\
&\geqslant (2+\dfrac{2}{n}+\dfrac{n-1}{3n})(n+1)T(r)-\dfrac{2(n+1)}{n}T(r)\\
&-(3n+3+\dfrac{n-1}{3})\sum_{i=1}^{2n+2}\dfrac{1}{k_i+1}T(r)+o(T(r)).
\end{align*}
Letting $r\longrightarrow +\infty$, we get
$$ 2n+2\geqslant (2+\dfrac{2}{n}+\dfrac{n-1}{3n})(n+1)-\dfrac{2(n+1)}{n}-(3n+3+\dfrac{n-1}{3})\sum_{i=1}^{2n+2}\dfrac{1}{k_i+1}.$$
Thus
$$ \sum_{i=1}^{2n+2}\dfrac{1}{k_i+1}\geqslant \dfrac{n^2-1}{10n^2+8n}$$
This is a contradiction. 

Hence the supposition is impossible. Therefore, $\sharp\mathcal F(f,\{H_i,k_i\}_{i=1}^{2n+2},1)\leqslant 2$. We complete the proof of the theorem.
\end{proof}

\begin{proof}[Proof of Theorem 1.4]
Let $f^1,f^2,f^3\in\mathcal F(f,\{H_i,k_i\}^{2n+1}_{i=1},p)$. Suppose that $f^1\times f^2\times f^3 :\C^m\rightarrow\P^n(\C)\times \P^n(\C)\times \P^n(\C)$ is linearly nondegenerate, where $\P^n(\C)\times \P^n(\C)\times \P^n(\C)$ is embedded into $\P^{(n+1)^3-1}(\C)$ by Seger imbedding. Then for every $s,t,l$ we have
$$ 
P:=Det\left (\begin{array}{ccc}
(f^1,H_s)&(f^1,H_t)&(f^1,H_l)\\ 
(f^2,H_s)&(f^2,H_t)&(f^2,H_l)\\
(f^3,H_s)&(f^3,H_t)&(f^3,H_l)
\end{array}\right )\not\equiv 0.
$$
By Lemma \ref{new1} we have
\begin{align*}
T(r)\geqslant&\sum_{i=s,t,l}(N(r,\min\{\nu_{(f^u,H_i),\leqslant k_i};1\leqslant u\leqslant 3\})\\
&-N^{(1)}_{(f,H_i),\leqslant k_i}(r))+ 2\sum_{i=1}^{2n+1}N^{(1)}_{(f,H_i),\leqslant k_i}(r)+o(T(r)),
\end{align*}
where $T(r)=\sum_{u=1}^3T_{f^u}(r)$. 
Summing-up both sides of the above inequality over all $(s,t,l)$, we obtain
\begin{align}\notag
T(r)\geqslant\dfrac{1}{2n+1}\sum_{i=1}^{2n+1}&(3N(r,\min\{\nu_{(f^u,H_i),\leqslant k_i};1\leqslant u\leqslant 3\})\\
\label{3.17}
&+(4n-1)N^{(1)}_{(f,H_i),\leqslant k_i}(r))+o(T(r)).
\end{align}
It is easy to see that for positive integers $a,b,c$ with $\min\{a,p\}=\min\{b,p\}=\min\{c,p\}$, we have
$$ 3\min\{a,b,c\}+(4n-1)\geqslant\dfrac{4n-1+3p}{2n+p}(\min\{a,n\}+\min\{b,n\}+\min\{c,n\}). $$
Hence
\begin{align*}
3N(r,\min\{\nu_{(f^u,H_i),\leqslant k_i};1\leqslant u\leqslant 3\})&+(4n-1)N^{(1)}_{(f,H_i),\leqslant k_i}(r)\\
&\geqslant \dfrac{4n-1+3p}{2n+p}\sum_{u=1}^3N^{(n)}_{(f,H_i),\leqslant k_i}(r),\ (1\leqslant i\leqslant 2n+1). 
\end{align*}
Therefore, the inequality (\ref{3.17}) implies that
\begin{align*}
T(r)&\geqslant\dfrac{1}{2n+1}\sum_{i=1}^{2n+1}\dfrac{4n-1+3p}{2n+p}\sum_{u=1}^3N^{(n)}_{(f,H_i),\leqslant k_i}(r)+o(T(r))\\
&\geqslant\dfrac{4n-1+3p}{(2n+1)(2n+p)}\sum_{i=1}^{2n+1}\sum_{u=1}^3(N^{(n)}_{(f,H_i)}(r)-N^{(n)}_{(f,H_i),> k_i}(r))+o(T(r))\\
&\geqslant\dfrac{4n-1+3p}{(2n+1)(2n+p)}(n-\sum_{i=1}^{2n+1}\dfrac{n}{k_i+1})T(r)+o(T(r)).
\end{align*}
Letting $r\longrightarrow +\infty$, we get
$$ 1\geqslant \dfrac{4n-1+3p}{(2n+1)(2n+p)}(n-\sum_{i=1}^{2n+1}\dfrac{n}{k_i+1}),$$
$$ i.e., \sum_{i=1}^{2n+1}\dfrac{1}{k_i+1}\geqslant\dfrac{np-3n-p}{4n^2+3np-n}. $$
This is a contradiction. 

Hence, the map $f^1\times f^2\times f^3$ is linearly degenerate. The theorem is proved.
\end{proof}

\end{document}